\documentclass[11pt,reqno]{amsart}
\usepackage{amssymb,latexsym}
\usepackage{graphicx}
\usepackage{url}		

\usepackage[low-sup]{subdepth}

\calclayout

\newcommand{\N}{{\mathbb N}}
\newcommand{\Z}{{\mathbb Z}}

\newtheorem{theorem}{Theorem}
\newtheorem{lemma}[theorem]{Lemma}
\newtheorem{corollary}[theorem]{Corollary}
\newtheorem*{conjecture}{Conjecture}
\newtheorem*{walltheorem}{Theorem}

\theoremstyle{definition}

\theoremstyle{remark}

\newcommand\fix{{\rm{Fix}}}
\newcommand\orbit{{\rm{Orb}}}
\newcommand\divides{{\mathchoice{\mathrel{\bigm|}}{\mathrel{\bigm|}}{\mathrel{|}}%
       {\mathrel{|}}}}
\newcommand\smalldivides{\mathrel{\kern-2pt\kern3.5pt|}}
\newcommand\notdivides{\mathrel{\kern-3pt\not\!\kern4.3pt\bigm|}}
\newcommand\smallnotdivides{\mathrel{\kern-2pt\not\!\kern3.5pt|}}
\def\le{\leqslant}
\def\ge{\geqslant}
\DeclareMathOperator\trace{Trace}

\begin{document}

\title{Fibonacci along even powers is (almost) realizable}
\author{Patrick Moss}
\address{School of Mathematics, University of East Anglia,
Norwich NR4 7TJ, U.K.}
\email{pbsmoss2@btinternet.com}
\thanks{}
\author{Tom Ward}
\address{School of Mathematics, University of Leeds,
Leeds LS2 9JT, U.K.}
\email{t.b.ward@leeds.ac.uk}

\begin{abstract}
An integer sequence is called realizable
if it is the count of periodic points of
some map. The Fibonacci sequence~$(F_n)$
does not have this property, and the Fibonacci
sequence sampled along the squares~$(F_{n^2})$
also does not have this property. We
prove that this is an
arithmetic phenomenon related to the discriminant
of the Fibonacci sequence, by showing that
the sequence~$(5F_{n^2})$ is realizable.
More generally, we show that~$(F_{n^{2k-1}})$
is not realizable in a particularly strong sense
while~$(5F_{n^{2k}})$ is realizable,
for any~$k\ge1$.
\end{abstract}

\maketitle

\section{Introduction}\label{sectionIntroduction}

Counting periodic points for iterates of maps
provides a natural source of integer sequences.
For example, the `golden mean' shift map~$\sigma$ on
\[
\Sigma=\{x\in\{0,1\}^{\Z}\mid(x_{k},x_{k+1})\neq(1,1)\}
\]
defined by~$\sigma\colon(x_n)_{n\in\Z}\mapsto(x_{n+1})_{n\in\Z}$
has
\[
\fix_n(\sigma)=\{x\in\Sigma\mid\sigma^n(x)=x\}
=
\trace\begin{pmatrix}{1\thinspace}{1}\\{1\thinspace}{0 }\end{pmatrix}^n=L_n
\]
for all~$n\in\N$,
giving the Lucas sequence~$(L_n)=(1,3,4,\dots)$.
The sequences we discuss start with the first term
(not the zeroth), and we use~$U=(U_n)$
for a sequence~$(U_n)_{n\in\N}$
where~$\N=\{1,2,3,\dots\}$.

We call an integer sequence~$(U_n)_{n\in\N}$ `realizable'
if there is some map~$T\colon X\to X$
with the property that~$\fix_n(T)=U_n$ for all~$n\ge1$.
Puri and Ward~\cite{MR1866354} proved that the golden mean shift map
illustrates a uniqueness phenomenon, by showing
that if~$(U_n)$ is an integer sequence satisfying the
Fibonacci recurrence~$U_{n+2}=U_{n+1}+U_{n}$ for all~$n\ge1$
with~$U_1=a$ and~$U_2=b$, then~$(U_n)$ is realizable
if and only if~$b=3a$ (meaning that~$(U_n)$ is
an integer multiple of the Lucas sequence~$(L_n)$).
In order to explain a more general
setting within which this is a special case,
we recall that a sequence~$(U_n)$ is realizable
if and only if it satisfies two conditions:
\begin{enumerate}
\item the Dold condition~\cite{MR724012}
that~$\sum_{d\smalldivides n}\mu\bigl(\tfrac{n}{d}\bigr)U_d
\equiv0$ modulo~$n$ for all~$n\in\N$, and
\item the sign condition~$\sum_{d\smalldivides n}\mu\bigl(\tfrac{n}{d}\bigr)U_d\ge0$
for all~$n\in\N$.
\end{enumerate}
All this means is that~$\fix_T(n)=U_n$ for all~$n\ge1$
if and only if the number of closed orbits of
length~$n$ under~$T$ is given by~$\orbit_T(n)
=\frac{1}{n}\sum_{d\smalldivides n}\mu\bigl(\tfrac{n}{d}\bigr)U_d$,
and so it must be the case that~$\orbit_T(n)$
is a non-negative integer
for all~$n\in\N$.
Minton~\cite{MR3195758} showed that any linear recurrence sequence
satisfying the Dold condition must be a sum of traces of
powers of algebraic numbers (and, in the case of binary
recurrences, the sign condition is easily understood). This
recovers the result of~\cite{MR1866354}, and much else
besides.

The Fibonacci sequence itself,~$(F_n)=(1,1,2,3,\dots)$,
is not realizable. Indeed, it fails the Dold condition in the following
strong---and, in the sense of Theorem~\ref{theoremFiveFibonacciSquared},
irreparable---way.

\begin{lemma}\label{lemmaInfinitelyManyPrimeDenominators}
The set of primes dividing a denominator
of~$\frac{1}{n}\sum_{d\smalldivides n}\mu\bigl(\tfrac{n}{d}\bigr)F_d$
for some~$n\in\N$ is infinite.
\end{lemma}

\begin{proof}
We recall that~$F_p$ is equivalent modulo~$p$
to the Legendre symbol~$\left(\frac{p}{5}\right)$
for any prime~$p$
(see Ribenboim~\cite[Eq.~(IV.13), p.~60]{MR1377060}
or Lemmermeyer~\cite[Ex.~2.25, p.~73]{MR1761696}).
It follows that if~$p$ is an odd
prime with~$p\equiv\pm2$ modulo~$5$,
then~$F_p\equiv-1$ modulo~$p$, so
the denominator of~$\tfrac{1}{p}(F_p-1)=\frac{1}{p}
\sum_{d\smalldivides p}\mu\bigl(\tfrac{p}{d}\bigr)F_d$ is~$p$.
\end{proof}

A numerical observation is that this seems to be typical
for linear recurrence sequences, in the
following sense. An integer linear recurrence sequence may
be realizable (and, up to understanding the sign condition,
Minton's results determine when this is the case), but if
it fails to be realizable then the denominators
appearing in the associated sequence whose non-integrality
witnesses the failure of realizability are expected to be divisible by
infinitely many primes.

In a different direction,
Moss~\cite{pm} showed that realizability is preserved by
a surprising diversity of `time-changes': that is, there are
non-trivial maps~$h\colon\N\to\N$ with the property
that if~$(U_n)$ is a realizable sequence, then~$(U_{h(n)})$ is
also realizable. Examples from~\cite{pm}
with this realizability-preserving property include the monomials,
and in later work Jaidee, Moss and Ward~\cite{MR4002553} showed
that the monomials are the only polynomials with this
property, and that there are uncountably many maps with
this property.

The unexpected phenomena we wish to discuss here is
that some of these time-changes that preserve realizability
seem to `repair' the failure to be realizable for the Fibonacci
sequence---up to a finite set of primes. At this stage we understand neither
the reason for this, nor its extent.

\begin{theorem}\label{theoremFiveFibonacciSquared}
The sequence~$(F_{n^{2}})$ is not realizable, but the
sequence~$(5F_{n^{2}})$ is.
\end{theorem}

The negative part of Theorem~\ref{theoremFiveFibonacciSquared}
may be seen from the observation
\[
\tfrac{1}{5}\sum_{d\smalldivides 5}\mu\bigl(\tfrac{5}{d}\bigr)F_{d^2}=
\tfrac{1}{5}\left(F_{25}-F_{1}\right)=
\tfrac{75024}{5},
\]
which shows that~$(F_{n^2})$ fails the Dold congruence for realizability.
The positive part of Theorem~\ref{theoremFiveFibonacciSquared}
consists of a direct proof that the sequence~$(5F_{n^2})$
satisfies the Dold conditions and the sign condition, and this
will require several steps.
Lemma~\ref{lemmaInfinitelyManyPrimeDenominators}
(strictly speaking, its proof), Theorem~\ref{theoremFiveFibonacciSquared},
and a result from~\cite{MR4002553}
together give the following description of
the behaviour of the Fibonacci sequence along
powers.

\begin{corollary}\label{corollary}
If~$j$ is odd, then
the set of primes dividing denominators
of~$\frac{1}{n}\sum_{d\smalldivides n}\mu\bigl(\tfrac{n}{d}\bigr)F_{d^j}$
for~$n\in\N$ is infinite.
If~$j$ is even, then the sequence~$(F_{n^j})$ is
not realizable, but the sequence~$(5F_{n^j})$ is.
\end{corollary}

\section{Modular periods of the Fibonacci sequence}

It will be convenient to use
Dirichlet convolution notation,
so that for sequences~$f=(f_n)$ and~$g=(g_n)$ we write
\[
(f*g)_n=\sum_{d\smalldivides n}f_dg_{n/d}
\]
for all~$n\ge1$. The two conditions for realizability
of a sequence~$U=(U_n)$ can then be stated as~$(\mu*U)_n\equiv0$
modulo~$n$ and~$(\mu*U)_n\ge0$ for all~$n\ge1$.

The argument involves working modulo various natural numbers,
and we adopt the convention that a representative of an
equivalence class modulo~$m\in\N$ is always chosen
among the representatives
\[
\{0,1,\dots,m-1\}.
\]
The sequence~$(F_n)$ is automatically
periodic modulo~$m$, and we define~$\ell(m)$ to be its period.
That is,
\begin{equation*}\label{firstequationinFibonacciSquared}
\ell(m)=\min\{d\in\N\mid F_{n+d}\equiv F_n\pmod{m}\mbox{ for all }n\in\N\}.
\end{equation*}
The quantity~$\ell(m)$ is well studied; a convenient
source for the type of results we need is the paper of Wall~\cite{MR120188},
whose Theorems~5,~6 and~7 give the following.

\begin{walltheorem}[Wall~\cite{MR120188}]
If~$p$ is an odd prime, then
\begin{align}
\ell(p)&\divides p-1\mbox{ if }p\equiv\pm1\pmod{10}\label{Wall1}
\intertext{and}
\ell(p)&\divides2(p+1)\mbox{ if }p\equiv\pm3\pmod{10}.\label{Wall2}
\end{align}
If~$p$ is a prime with~$\ell(p)\neq\ell(p^2)$,
then
\begin{equation*}\label{Wall3}
\ell(p^n)=p^{n-1}\ell(p)
\end{equation*}
for all~$n\in\N$.
Moreover, if~$t$ is the largest integer
with~$\ell(p^t)=\ell(p)$, then
\begin{equation}\label{Wall4}
\ell(p^n)=p^{n-t}\ell(p)
\end{equation}
for all~$n\in\N$ with~$n\ge t$.
\end{walltheorem}

From now on in this section~$p$ will always denote a prime,
and~$k$ an integer with~$k\ge2$.

Clearly~\eqref{Wall1} and~\eqref{Wall2}
show that~$\ell(p)\divides 2(p^2-1)$ for an
odd prime~$p$, but a little more is true.
We claim that
\begin{equation}\label{Wall5}
\ell(p)\divides p^2-1
\end{equation}
for any prime~$p$.
For~$p=2$ it is easy to check that~$\ell(p)=3$.
For an odd prime~$p\equiv\pm1$ modulo~$10$,~\eqref{Wall1}
shows that~$\ell(p)\divides p^2-1$.
For an odd prime~$p\equiv\pm3$ modulo~$10$,~$p-1$
is even so~$p^2-1$ is a multiple of~$2(p+1)$,
and hence~$\ell(p)\divides p^2-1$ by~\eqref{Wall2}.

By definition,~$F_{n+\ell(p^d)}\equiv F_n$ modulo~$p^d$
and~$F_{n+\ell(p^{d+1})}\equiv F_n$ modulo~$p^{d+1}$
for any~$d\in\N$, so~$F_{n+\ell(p^{d+1})}\equiv F_n$ modulo~$p^d$
and hence
\begin{equation}\label{Wall6}
\ell(p^d)\le\ell(p^{d+1})
\end{equation}
for any~$d\in\N$.

\begin{lemma}\label{lemmaDefiningsofn1}
For~$n\in\N$ there is
some~$s=s(n)$ with~$0\le s<n$
such that~$\ell(p^n)=p^s\ell(p)$.
\end{lemma}

\begin{proof}
If~$\ell(p)\neq\ell(p^2)$ then~\eqref{Wall4}
gives~$s=n-1$.

Suppose therefore that~$\ell(p)=\ell(p^2)$.
If~$\ell(p)=\ell(p^n)$ for all~$n\in\N$,
then we may set~$s=0$.
If~$\ell(p)\neq\ell(p^n)$ for some~$n\in\N$,
then let~$t\in\N$ be the largest integer
with~$\ell(p^t)=\ell(p)$.
By~\eqref{Wall6} we then have~$\ell(p)=\ell(p^j)$
for~$j=1,\dots,t$, and so we can use~\eqref{Wall4}
to define~$s=0$ if~$n\le t$ and~$s=n-t$ if~$n>t$.
\end{proof}

By Lemma~\ref{lemmaDefiningsofn1} we have
\[
\frac{p^n(p^2-1)}{\ell(p^n)}=\frac{p^{n-s}(p^2-1)}{\ell(p)},
\]
so~\eqref{Wall5} shows that
\begin{equation}\label{Wall7}
\ell(p^n)\divides p^n(p^2-1)
\end{equation}
for any~$n\in\N$.

We now define sequences~$u=(u_n)$
and~$v=(v_n)$ by
\begin{align}
u&=(F_n\pmod{p^{2k}})\nonumber
\intertext{and}
v&=(F_n\pmod{p^{2(k-1)}}).\label{Wall9}
\end{align}

\begin{lemma}\label{lemmaPeopleNeedToKnowYoureIndependent}
For any integer~$c\ge0$ we have
\begin{equation*}
F_{p^{2k}+c}\equiv F_{p^{2(k-1)}+c}
\end{equation*}
modulo~$p^k$.
\end{lemma}

\begin{proof}
By definition,
\begin{align*}
F_{p^{2k}+c}&\equiv u_{p^{2k}+c}\pmod{p^{2k}}
\intertext{and}
F_{p^{2(k-1)}+c}&\equiv v_{p^{2(k-1)}+c}\pmod{p^{2(k-1)}}
\end{align*}
For any integer~$j\ge0$ we have
\[
v_{p^{2(k-1)}+c}=v_{p^2(k-1)+c+j\ell(p^{2(k-1)})}.
\]
By~\eqref{Wall7} we may set~$j=\frac{p^{2(k-1)(p^2-1)}}{\ell(p^{2(k-1)})}$,
so
\[
v_{p^{2(k-1)}+c+j\ell(p^{2(k-1)})}
=
v_{p^{2(k-1)}+c+p^{2(k-1)}(p^2-1)}
=
v_{p^{2k}+c}.
\]
It follows that~$v_{p^{2(k-1)}+c}=v_{p^{2k}+c}$
and so
\[
F_{p^{2(k-1)}+c}\equiv v_{p^{2k}+c}\pmod{p^{2(k-1)}}.
\]
Clearly~$p^{2(k-1)}\divides u_n-v_n$ for all~$n\in\N$,
so~$p^k\divides u_n-v_n$ for all~$n\in\N$ since~$k\ge2$.
In particular,
\begin{align*}
p^k&\divides u_{p^{2k}+c}-v_{p^{2k}+c}
\intertext{and hence}
p^k&\divides u_{p^{2k}+c}-v_{p^{2(k-1)}+c}.
\end{align*}
Thus~$F_{p^{2k}+c}\equiv F_{p^{2(k-1)}+c}$ modulo~$p^k$
as required.
\end{proof}

Notice that in the proof above we
saw that~$p^{2(k-1)}\divides u_n-v_n$,
so for~$k\ge3$ we have~$p^{k+1}\divides u_n-v_n$.
It follows that
\begin{equation}\label{Wall10}
F_{2^{2k}+c}\equiv F_{2^{2(k-1)}+c}\pmod{2^{k+1}}
\end{equation}
for any~$k\ge3$.

\section{Properties of the sequence $(5F_{n^2})$}

In this section~$p$ again denotes a prime,~$k$
a positive integer,~$\phi=(\phi_n)$ denotes
the sequence defined by~$\phi_n=5F_{n^2}$
for all~$n\in\N$, and~$L=(L_n)$
denotes the Lucas sequence.
Since~$L$ is a realizable sequence
it satisfies the Dold congruences, and so
in particular we have
\begin{equation}\label{Wall11}
L_p\equiv1\pmod{p}.
\end{equation}
The next result appeared
as an exercise due to Desmond~\cite{desmond1970},
with a solution using an earlier result of
Ruggles~\cite{ruggles}.

\begin{lemma}[Desmond]\label{lemmaDesmond}
For a non-negative integer~$n$ we have~$F_{np}\equiv F_nF_p$
modulo~$p$.
\end{lemma}

\begin{proof}
The case~$p=2$ or~$n\le1$ is clear, so assume that~$p$ is odd
and~$n\ge2$,
and assume the statement holds for all~$n\le m$
for some~$m\ge2$.
Recall that~$F_{r+s}=F_rL_s+(-1)^{s+1}F_{r-s}$
(see, for example, Ribenboim~\cite[Eq.~(2.8)]{MR1761897}).
It follows that
\[
F_{mp+p}=F_{mp}L_p+(-1)^{p+1}F_{mp-p}=F_{mp}L_p+F_{(m-1)p},
\]
and so
\[
F_{(m+1)p}\equiv F_{mp}+F_{(m-1)p}\pmod{p}
\]
by~\eqref{Wall11}.
The inductive assumption gives
\[
F_{(m+1)p}\equiv F_mF_p+F_{m-1}F_p\pmod{p},
\]
and then the relation~$F_mF_p+F_{m-1}F_p=F_{m+1}F_p$
completes the proof by induction.
\end{proof}

By Lemma~\ref{lemmaDesmond} we have
\[
F_{np^2}\equiv F_{np}F_p\equiv F_n(F_p)^2\pmod{p}.
\]
Since~$F_p^2\equiv\bigl(\frac{p}{5}\bigr)\equiv1$
modulo~$p$ if~$p\neq5$, we deduce that
\begin{equation*}\label{Wall12}
F_{np^2}\equiv F_n\pmod{p}
\end{equation*}
for~$p\neq5$.
It follows that
\begin{equation}\label{Wall13}
5F_{np^2}\equiv5F_n\pmod{p}
\end{equation}
for any prime~$p$,
since it is trivial for~$p=5$.

\begin{lemma}\label{lemmaWasPatsTh3.4}
For non-negative integers~$n$ and~$k$
we have
\[
5F_{np^{2k}}\equiv5F_{np^{2(k-1)}}\pmod{p^k}.
\]
\end{lemma}

\begin{proof}
For~$k=1$ this follows from~\eqref{Wall13}.
If~$k>1$, then Lemma~\ref{lemmaPeopleNeedToKnowYoureIndependent}
with~$c=(n-1)p^{2k}$ gives
\begin{equation}\label{Wall14}
F_{np^{2k}}\equiv F_{p^{2(k-1)}+(n-1)p^{2k}}\pmod{p^k}.
\end{equation}
By the definition~\eqref{Wall9} we have
\[
F_{p^{2(k-1)}+(n-1)p^{2k}}
\equiv
v_{p^{2(k-1)}+(n-1)p^{2k}}\pmod{p^{2(k-1)}},
\]
and for~$j\ge0$ we have
\[
v_{np^{2(k-1)}}
=
v_{p^{2(k-1)}+(n-1)p^{2(k-1)}+j\ell(p^{2(k-1)})}.
\]
Taking~$j=\frac{(n-1)p^{2(k-1)}(p^2-1)}{\ell(p^{2(k-1)})}$,
which is integral by~\eqref{Wall7}, this gives
\[
v_{p^{2(k-1)}+(n-1)p^{2(k-1)}+j\ell(p^{2(k-1)})}
=
v_{p^{2(k-1)}+(n-1)p^{2k}}
\]
so~$v_{p^{2(k-1)}+(n-1)p^{2k}}=v_{np^{2(k-1)}}$.
It follows that
\begin{align*}
F_{p^{2(k-1)}+(n-1)p^{2k}}
&\equiv
F_{np^{2(k-1)}}\pmod{p^{2(k-1)}}.
\intertext{Since~$k>1$, this gives}
F_{p^{2(k-1)}+(n-1)p^{2k}}
&\equiv
F_{np^{2(k-1)}}\pmod{p^{k}}
\intertext{and hence}
F_{np^{2k}}
&\equiv
F_{np^{2(k-1)}}\pmod{p^{k}}
\end{align*}
by~\eqref{Wall14}.
\end{proof}

A similar argument using~\eqref{Wall10}
shows that if~$n$ is a positive integer
and~$k\ge3$ then
\begin{equation*}\label{Wall17}
F_{n2^{2k}}\equiv F_{n2^{2(k-1)}}\pmod{2^{k+1}}.
\end{equation*}

The modular arguments thus far are of course aimed
at establishing the Dold condition. The sign
condition is satisfied because of the rapid
rate of growth in the sequence, which is more
than sufficient by the following remark
of Puri~\cite{yash}.

\begin{lemma}\label{lemmaYashGrowthRateThing}
If~$(A_n)$ is an increasing sequence of non-negative
real numbers with~$A_{2n}\ge nA_n$ for all~$n\in\N$,
then~$(\mu* A)_n\ge0$
for all~$n\in\N$.
\end{lemma}

\begin{proof}
In the even case, we have
\begin{align*}
(\mu*A)_{2n}
=
\sum_{d\smalldivides2n}\mu(2n/d)A_d
\ge
A_{2n}-\sum_{k=1}^{n}A_k
\ge
A_{2n}-nA_n\ge0,
\end{align*}
since the largest divisor of~$2n$ is~$n$.
Similarly, in the odd case we have
\begin{align*}
(\mu*A)_{2n+1}
\ge
A_{2n+1}-\sum_{k=1}^{n}A_k\ge A_{2n}-nA_n\ge0,
\end{align*}
since the largest divisor of~$2n+1$ is smaller than~$n$,
proving the lemma.
\end{proof}

\begin{proof}[Proof of the positive part of Theorem~\ref{theoremFiveFibonacciSquared}.]
We wish to show that~$n\divides(\mu*\phi)_n$ and~$(\mu*\phi)_n\ge0$
for all~$n\in\N$.
For~$n=1$ this is clear. If~$n=p^k$ then
\[
(\mu*\phi)_n
=
\sum_{d\smalldivides p^k}\mu(d)\phi_{p^k/d}
=
\phi_{p^k}-\phi_{p^{k-1}}
=
5F_{p^{2k}}-5F_{p^{2(k-1)}},
\]
which is clearly non-negative, and
Lemma~\ref{lemmaWasPatsTh3.4}
shows that it
is divisible by~$n$.

For the general case we will work with one
prime at a time using Lemma~\ref{lemmaWasPatsTh3.4}.
Suppose that~$n=p_1^{k_1}\cdots p_m^{k_m}$
with~$m\ge2$,~$k_1,\dots,k_m\in\N$, and distinct primes~$p_1,\dots,p_m$.
Select one of these primes~$p_i$, and
to reduce the notational complexity write~$p^k=p_i^{k_i}$.
Writing~$s=n/p_i^{k_i}$,
we have
\[
(\mu*\phi)_n
=
\sum_{d\smalldivides p^ks}\mu(d)\phi_{p^ks/d}
=
\sum_{d\smalldivides s}
\bigl(\phi_{p^ks/d}-\phi_{p^{k-1}s/d}\bigr)
=
\sum_{d\smalldivides s}\bigl(5F_{(s/d)^2p^{2k}}-5F_{(s/d)^2p^{2(k-1)}}\bigr).
\]
Lemma~\ref{lemmaWasPatsTh3.4} therefore shows that~$p_i^{k_i}\divides(\mu*\phi)_n$,
and, by using this for each prime dividing~$n$, we deduce
that~$n\divides(\mu*\phi_n)$ as required.

For the sign condition we will use Lemma~\ref{lemmaYashGrowthRateThing}
and Binet's formula.
Clearly~$\phi$ is an increasing sequence.
Writing~$\alpha=(1+\sqrt{5})/2$
and~$\beta=(1-\sqrt{5})/2$
we have
\[
\phi_{2n}=5F_{4n^2}
=
\sqrt{5}
\bigl(\alpha^{2n^2}+\beta^{2n^2}\bigr)
\bigl(\alpha^{n^2}+\beta^{n^2}\bigr)
\bigl(\alpha^{n^2}-\beta^{n^2}\bigr)
\]
and
\[
n\phi_n=n\sqrt{5}\bigl(\alpha^{n^2}-\beta^{n^2}\bigr).
\]
Thus to show the growth condition used in
Lemma~\ref{lemmaYashGrowthRateThing}
it is enough to show that
\[
\bigl(\alpha^{2n^2}+\beta^{2n^2}\bigr)
\bigl(\alpha^{n^2}+\beta^{n^2}\bigr)\ge n
\]
for~$n\in\N$.
Clearly
\[
\bigl(\alpha^{2n^2}+\beta^{2n^2}\bigr)
\bigl(\alpha^{n^2}+\beta^{n^2}\bigr)
>
G(n)=\alpha^{2n^2}\bigl(\alpha^{n^2}-1\bigr)
\]
for all~$n\in\N$.
We check that~$G(1)=\alpha>1$ and for~$n\ge2$
we have
\[
G(n)>\alpha^{2n^2}(\alpha-1)=\alpha^{2n^2-1}>n.
\]
Thus~$\phi_{2n}\ge n\phi_n$ for all~$n\in\N$,
completing the proof.
\end{proof}

\begin{proof}[Proof of Corollary~\ref{corollary}.]
Assume first that~$j$ is odd, and
recall that if~$p\equiv\pm2$
modulo~$5$ is an odd prime, then~$F_p\equiv-1$ modulo~$p$
(as in the proof of Lemma~\ref{lemmaInfinitelyManyPrimeDenominators}).
By Lemma~\ref{lemmaDesmond} it follows that~$F_{p^j}\equiv-1$
modulo~$p$, so
the denominator of~$\tfrac{1}{p}(F_p-1)=\frac{1}{p}
\sum_{d\smalldivides p}\mu\bigl(\tfrac{p}{d}\bigr)F_d$ is~$p$.

For~$j$ even, Lemma~\ref{lemmaDesmond} shows that~$F_{5^j}\equiv F_{25}\equiv0$
modulo~$5$, so~$(1/5)\bigl(F_{5^j}-F_{1}\bigr)$ has denominator~$5$,
showing that~$(F_{n^j})$ is not realizable.

Finally, by~\cite[Th.~5]{MR4002553} we know that
for any~$k\in\N$ the map~$h(n)=n^k$ preserves
realizability. That is, if~$(U_n)$ is a realizable
sequence then~$(U_{n^k})$ is also. Thus
the positive part of Theorem~\ref{theoremFiveFibonacciSquared}
shows that~$(5F_{n^{2k}})$ is realizable
for any~$k\in\N$.
\end{proof}

\section{Remarks}

\noindent(1) The correspondence between a pair~$(X,T)$, denoting
a map~$T\colon X\to X$ with the property that~$\fix_n(T)<\infty$ for
all~$n\ge1$, and the associated sequence~$(\fix_n(T))$
or~$(\orbit_n(T))$ is `functorial' with regard to many
natural operations (we refer to work of Pakapongpun
and Ward~\cite{MR2486259,MR3194906} for an explanation of this cryptic comment,
and for results in this direction).
Unfortunately the time-changes studied in~\cite{MR4002553}
do not seem to have any such property. For example,
we do not have any reasonable way to start with
a pair~$(X,T)$ and set-theoretically `construct'
another pair~$(X',T')$ with the property that~$\fix_n(T')=\fix_{n^2}(T)$
for all~$n\ge1$.
We have even less ability---indeed, have no starting
point---to `construct' some reasonable
pair~$(X,T)$ with~$\fix_n(T)=5F_{n^2}$ for all~$n\ge1$,
particularly if the permutation of a countable set
implicitly constructed in the proof is viewed as
unreasonable. A general result due to Windsor~\cite{MR2422026}
shows that there must be a~$C^{\infty}$ map of
the~$2$-torus with this property, but
we know nothing more meaningful about such a map
than the fact that it must exist.

\noindent(2) For integers~$P,Q$ we define the Lucas
sequence~$(U_n(P,Q))$ and companion Lucas sequence~$(V_n(P,Q))$
by
\begin{align*}
\frac{x}{1-Px+Qx^2}&=\sum_{n=0}^{\infty}U_n(P,Q)x^n
\intertext{and}
\frac{2-Px}{1-Px+Qx^2}&=\sum_{n=0}^{\infty}V_n(P,Q)x^n.
\end{align*}
Binet's formulas show that the sequence~$(V_n(P,Q))$
always satisfies
the Dold condition, but that~$(U_n(P,Q))$ can only
do so if the discriminant~$P^2-4Q=\pm1$.
Thus the sequence
\[
\bigl(U_n(\pm(2k+1),k^2+k)\bigr)
\]
satisfies the Dold condition for any~$k\in\Z$.
Theorem~\ref{theoremFiveFibonacciSquared}
states that~$(5U_{n^2}(1,-1))$ is realizable,
and numerical experiments suggest the following.

\begin{conjecture}
For~$P,Q\in\Z$ the sequence~$\bigl((P^2-4Q)U_{n^2}(P,Q)\bigr)$
satisfies the Dold condition.
\end{conjecture}


\providecommand{\bysame}{\leavevmode\hbox to3em{\hrulefill}\thinspace}
\providecommand{\MR}{\relax\ifhmode\unskip\space\fi MR }
\providecommand{\MRhref}[2]{%
  \href{http://www.ams.org/mathscinet-getitem?mr=#1}{#2}
}
\providecommand{\href}[2]{#2}

\end{document}